%% file: ppartrefed2.tex
\documentclass[reqno,12pt]{amsart}

\usepackage[letterpaper]{geometry}
\usepackage{tikz}
\usepackage{enumerate}
\usepackage{graphicx}
\usepackage{amssymb}
\usepackage{epstopdf}
\usepackage{subcaption}
\usepackage{cite}
\newcommand{\R}{\mathbb{R}}

\newcommand{\C}{\mathbb{C}}
\newcommand{\Z}{\mathbb{Z}}

\newcommand{\cO}{\mathcal{O}}

\makeatletter
\newcommand{\lang}{\langle}
\newcommand{\rang}{\rangle}

\newcommand{\x}{{\bf x}}

\newtheorem{thm}{Theorem}[section]
\newtheorem{cor}[thm]{Corollary}
\newtheorem{lem}[thm]{Lemma}
\newtheorem*{lem*}{Lemma}

\theoremstyle{remark}
\newtheorem*{rem*}{Remark}

\theoremstyle{definition}

\date{\today}
\author{Holley Friedlander}
\address{Dickinson College}
\email{friedlah@dickinson.edu}
\title{On the $p$-parts of Weyl group multiple Dirichlet series}

\begin{document}

\baselineskip=17pt

\begin{abstract}
We study the structure of $p$-parts of Weyl group multiple Dirichlet series. In particular, we extend results of Chinta, Friedberg, and Gunnells \cite{ppart} and show, in the stable case, that the $p$-parts of Chinta and Gunnells \cite{wmnds} agree with those constructed using the crystal graph technique of Brubaker, Bump, and Friedberg \cite{wmds2,wmds4}. In this vein, we give an explicit recurrence relation on the coefficients of the $p$-parts, which allows us to describe the support of the $p$-parts and address the extent to which they are uniquely determined.
\end{abstract}

\subjclass[2010]{Primary 11F68, 11M41; Secondary 20F55}

\keywords{Weyl group multiple Dirichlet series, stable case, metaplectic}

\maketitle

\section{Introduction}
\label{sec:introduction}
Given a positive integer $n$, a global field $K$ containing all $2n$th roots of unity, and an irreducible, reduced root system $\Phi$ of rank $r$, Weyl group multiple Dirichlet series are Dirichlet series in $r$ complex variables, with analytic continuation to $\C^r$ and a group of functional equations isomorphic to the Weyl group of $\Phi$. Such series have applications in analytic number theory and arise, for example, in the study of moments of $L$-functions \cite{GH}.

We will be concerned with a class of Weyl group multiple Dirichlet series that arise as Fourier-Whittaker coefficients of metaplectic Eisenstein series (Eisenstein series on covers of reductive groups). Because the coefficients involve Gauss sums, in general, such series are {\em not} Eulerian. However, they do satisfy a twisted analogue of an Euler product:~each prime $p$ of $K$ corresponds to a local factor, and like classical Euler factors, the local $p$-parts are generating functions for the $p$-power coefficients of the series. In particular, the $p$-parts completely determine the global series.

Due to the complexity inherent in computing on metaplectic groups, significant effort has been devoted to the combinatorial problem of how to define $p$-parts in a way such that the resulting series coincide with those coming from Eisenstein series. At present, there are two main approaches for which the resulting series have been shown to yield the desired analytic properties of meromorphic continuation and Weyl group of functional equations. The first of Brubaker, Bump, and Friedberg defines the $p$-parts term by term via Gelfand-Tsetlin patterns, crystal graphs, or related combinatorial devices. This approach produces the desired analytic properties for all $\Phi$ when $n$ is sufficiently large \cite{wmds2,wmds4}, for all $n$ when $\Phi=A_r$ \cite{typea}, and in various other cases for classical types, e.g.~\cite{FZ}. The second of Chinta and Gunnells defines the $p$-parts all at once with an averaging technique analogous to the Weyl character formula. This approach produces a global object with the desired analytic properties for all $\Phi$ and all $n$ \cite{wmnds}. 

In this paper we study the structure of $p$-parts of Weyl group multiple Dirichlet series defined using the Chinta--Gunnells technique \cite{wmnds}. Our main goal is to show that the methods of \cite{wmds2,wmds4} and \cite{wmnds} give rise to the same series, in the situation where both are applicable. For $\Phi=A_r$ and $n=1$ the result follows from Tokuyama's formula \cite{Tokuyama}, which expresses a deformation of a Weyl character as a sum over a crystal graph. For $\Phi=A_r$ and general $n$, the result follows from the the combined works of Chinta and Offen \cite{chintaoffen} and McNamara \cite{mcnamara}, who compare the two methods via local Whittaker functions. For $\Phi$ simply laced and $n=2$, Chinta, Friedberg, and Gunnells \cite{ppart} provide further evidence by showing that the stable (see Section \ref{sec:stable}) coefficients of the local $p$-parts agree. We extend the results of \cite{ppart} to all $\Phi$ and $n$.

Our main tool to understand the structure of the $p$-parts is an explicit set of recurrence relations on the coefficients, which we state in Theorem \ref{thm:rec}. These relations allow us to prove Theorem \ref{thm:support}, which extends \cite[Theorem 3.2]{ppart} on the support of the $p$-parts to all $\Phi$ and all $n$. With Theorem \ref{thm:support}, we compare the stable coefficients of \cite{wmds4,wmds2} and \cite{wmnds}; Theorem \ref{thm:stable} extends \cite[Theorem 4.1]{ppart} and shows they do indeed agree.  Finally, Theorem \ref{thm:unstab} addresses the extent to which the recurrence relations uniquely determine the $p$-parts.

In addition to analyzing the $p$-parts, we explore the structure of modified $p$-parts (obtained by multiplying the $p$-parts with a rational factor) that have applications to the case when $K$ is a function field over a finite field. When $K$ is the rational function field, up to a variable change the modified $p$-parts coincide with the global Weyl group multiple Dirichlet series. This phenomenon was first noticed in \cite{C,mohler} and mirrors the similarity between the zeta function of the rational function field and its Euler factors.  As a complement to Theorem \ref{thm:support}, Theorem \ref{thm:gap} describes the support of these modified $p$-parts.  Just as in the case of the zeta function of a function field, we expect the associated multiple Dirichlet series to encode much information about the arithmetic of the defining curve. The first step is to understand the support of the series, which the author will pursue in a forthcoming paper. See also \cite{mydis}.

This paper proceeds as follows. In Section \ref{sec:prelims}, we review the Chinta--Gunnells construction of the $p$-parts \cite{wmnds}. In Section \ref{sec:rec}, we derive an explicit set of recurrence relations on the coefficients of the $p$-parts, allowing us to prove Theorem \ref{thm:support} on the support of the $p$-parts. At the end of Section \ref{sec:rec} we also prove Theorem \ref{thm:gap} on the support of the modified $p$-parts. Section \ref{sec:stable} compares the stable coefficients of \cite{wmds2,wmds4} and \cite{wmnds}, and Section \ref{sec:unstable} details what is known about the unstable coefficients and the extent to which the $p$-parts are uniquely determined.

\section{Preliminaries}
\label{sec:prelims}
The $p$-parts of Weyl group multiple Dirichlet series are built from combinatorial and number theoretic data. In this section, we review essential definitions related to root systems and Gauss sums and then describe the construction of the $p$-parts. 

\subsection{Root systems}
Let $\Phi$ be an irreducible, reduced root system of rank $r$ and $\{\alpha_1,\ldots, \alpha_r\}$ the simple roots. The root lattice $\Lambda$ of $\Phi$ is the $\Z$-span of the simple roots. For $\alpha=\sum k_i\alpha_i\in \Lambda$, define the generalized height function $d:\Lambda \to \Z$ as
\[
d:\alpha \mapsto \sum k_i.
\]

Denote by $W$ the Weyl group of $\Phi$, and choose a $W$-invariant symmetric bilinear inner product $(\cdot,\cdot)$ on $\Lambda \otimes \R$ normalized so that all short roots have length one\footnote{Our choice of inner product is different from that of our main reference \cite{ppart}, which follows the standard Bourbaki convention such that the short roots have length 2, except in the case when $\Phi=B_r$, in which the short roots are taken to have length 1.}. This normalization implies that for any $\alpha,\beta \in \Lambda$, we have $(\alpha,\beta) \in \frac{1}{2}\Z$.  For $\alpha,\beta \in \Lambda$, further define $\lang\alpha,\beta\rang=2\frac{(\alpha,\beta)}{(\beta,\beta)}$. Then the simple reflections 
\begin{equation}
\label{eqn:sigj}
\sigma_{\alpha_j}(\alpha_i)=\sigma_j(\alpha_i):=\alpha_i-\lang\alpha_i,\alpha_j\rang\alpha_j,
\end{equation}
generate $W$ and the Cartan matrix $(c_{ij})$ of $\Phi$ is defined by $c_{ij}:=\lang\alpha_i,\alpha_j\rang$.

Define the simple coroots $\check{\alpha}_i=2\alpha_i/(\alpha_i,\alpha_i)$ for $i=1,\ldots, r$. The dual lattice to $\Lambda$ with respect to the coroot basis is the weight lattice $L$ of $\Phi$ whose basis is given by the fundamental weights $\{\varpi_1,\ldots,\varpi_r\}$ defined by $(\varpi_i,\check{\alpha}_j)=\delta_{ij}$. There is a partial order on the weights: for $\lambda, \mu \in L$, we say $\lambda \succeq \mu$ if $\lambda-\mu=\sum k_i\varpi_i$ with all $k_i$ nonnegative. The weight $\lambda \in L$ is said to be dominant if $\lang \lambda,\alpha_i\rang \geq 0$ for all $i=1,\ldots, r$ and strongly dominant if the inequality is strict. For instance $\rho:=\sum_{i} \varpi_i$ is strongly dominant.

\subsection{Gauss sums}
Let $n$ be a positive integer and $K$ a global field containing all $2n$th roots of unity. Denote the ring of integers of $K$ by $\mathcal{O}_K$. Let $p$ be a prime of $K$ of norm $|p|$. For any $c \in \mathcal{O}_K$ and $t \in \Z$, one may associate a Gauss sum $g_t(c,p) \in \C$ modulo $p$. Gauss sums have a rich arithmetic structure, but here we will only require that they are complex numbers that satisfy the relations below. The interested reader may consult \cite[Section 2]{Kubota} for precise definitions. In particular we note that the Gauss sums we describe are generalizations of traditional finite field Gauss sums, whose properties are detailed in, for example, \cite[Chapter 8]{irelandrosen}. We have

\begin{align}
\label{eqn:gauss}
g_t(p^k,p^l)&=\left\{\begin{array}{cl}|p|^kg_{tl}(1,p)&\mbox{if } l=k+1;\\
\phi(p^l)&\mbox{if } n|tl \mbox{ and } k \geq l;\\
0&\mbox{otherwise}.\end{array}\right.
\end{align}
where $\phi(p^l)$ is the size of the residue field associated with $p^l$. Further, define 
\begin{equation}
\label{eqn:gauss3}
g_t(1,p)=-1\hspace{1cm}\mbox{if $n$ divides $t$.}
\end{equation}
Finally, if $t$ does not vanish modulo $n$, then 
\begin{equation}\label{eqn:gauss2}
g_t(1,p)g_{-t}(1,p)=|p|.
\end{equation}
As we will most often deal with the situation when $c=1$, we denote $g_t(1,p)$ by $g_t(p)$.

\subsection{Construction of $p$-parts}
\label{sec:cg}

We now review the Chinta--Gunnells construction \cite{wmnds} of $p$-parts of Weyl group multiple Dirichlet series. The idea is to use averaging in a way analogous to the Weyl character formula to define a rational function supported on $\Lambda$ that is invariant under a certain Weyl group action. We first define this action and then explain the relationship between the $p$-parts and the invariant function.
c
Fix an $r$-tuple of nonnegative integers $\ell=(l_1,\ldots,l_r)$. The tuple $\ell$ is a {\em twisting parameter} that determines a strongly dominant weight
\(
\theta=\theta(\ell):=\sum_{i=1}^r (l_i+1)\varpi_i
\)
(also referred to as the twisting parameter) and a $W$-action on $\Lambda$:
\begin{equation}
\label{eqn:bullet}
w\bullet\lambda=w(\lambda-\theta)+\theta.
\end{equation}
In particular, when $w=\sigma_j$ is a simple reflection, we have 
\(
\sigma_j\bullet\lambda=\sigma_j\lambda+(l_j+1)\alpha_j.
\)

Consider $A=\C[\Lambda]$, the ring of Laurent polynomials on $\Lambda$. The ring $A$ consists of all expressions $f$ of the form $f=\sum_{\beta \in \Lambda} c_\beta \x^\beta$ with $c_\beta \in \C$ almost all zero. Multiplication in $A$ is defined by addition in $\Lambda$: $\x^\beta\x^\lambda=\x^{\lambda+\beta}$ and we identify $A$ with $\C[x_1,x_1^{-1}\ldots,x_r,x_r^{-1}]$ via $\x^{\alpha_i}\mapsto x_i$. 

Our goal is to define a Weyl group action on the field of fractions $\tilde{A}$ of $A$. First, define a change of variables action on $A$ by 
\[
\left(\sigma_j(\x)\right)_i=|p|^{-c_{ij}}x_ix_j^{-c_{ij}}.
\]
This action is essentially a reformulation of the standard action of $W$ on $\Lambda$. One can check that if $f_\beta(\x)=\x^\beta$ is a monomial, then
\(
f_\beta(w\x)=|p|^{d(w^{-1}\beta-\beta)}\x^{w^{-1}\beta}.
\)

For $\alpha \in \Phi$, let
\(
n(\alpha)=n/\gcd(n,\|\alpha\|^2)\), and consider the sublattice $\Lambda'\subset \Lambda$ generated by the set $\{n(\alpha)\alpha\}_{\alpha \in \Phi}$. Define $\tilde{A}_\lambda$ as the set of functions $f/g \in \tilde{A}$ such that the support of $g$ lies in the kernel of the map $\nu:\Lambda \to \Lambda/\Lambda'$ and $\nu$ maps the support of $f$ to $\lambda$. Then we have the decomposition
\[
\tilde{A}=\bigoplus_{\lambda \in \Lambda/\Lambda'}\tilde{A}_\lambda.
\]

We are now ready to define the Chinta--Gunnells action. Fix $k \in \{1,\ldots, r\}$. For $\lambda \in \Lambda$, define \(\delta_{\ell,k}=\delta_k(\lambda):=d(\sigma_k\bullet\lambda-\lambda)\). Let $g^\ast_t(p)$ be the normalized Gauss sum
    \begin{equation}\label{gdef}
    g^\ast_t(p)=\left\{\begin{array}{ll}-1 &\mbox{ if } t\equiv 0\mod{n},\\
    g_t(p)/|p| &\mbox{ otherwise.}\end{array}\right.\end{equation}
For positive integers $a$ and $m$, let $(a)_m:=a-m\lfloor a/m\rfloor$ be the remainder of $a$ upon division by $m$. Also define $(-a)_m=0$ if $(a)_m=0$ and $(-a)_m=m-(a)_m$ otherwise.  Define the following rational functions:
    \begin{align}
    \label{eqn:locPQ}
    P_{\beta,\ell,k}(x_k)&=(|p| x_k)^{l_k+1-(\delta_k(\beta))_{n(\alpha_k)}}\frac{1-1/|p|}{1-(|p| x_k)^{n(\alpha_k)}/|p|},\\
    \label{eqn:locQ}
    Q_{\beta,\ell,k}(x_k)&=-g^\ast_{-\|\alpha_k\|^2\delta_k(\beta)}(p)(|p| x_k)^{l_k+1-{n(\alpha_k)}}\frac{1-(|p| x_k)^{n(\alpha_k)}}{1-(|p| x_k)^{n(\alpha_k)}/|p|}.\nonumber
    \end{align}
\noindent Then the simple reflection $\sigma_k\in W$ acts on $f(\x)\in A_\beta$ by
\begin{equation}
\label{eqn:act}
(f|_\ell \sigma_k)(\x)=(P_{\beta,\ell,k}(x_k)+Q_{\sigma\bullet\beta,\ell,k}(x_k))f(\sigma_k\x)
\end{equation}
and this action respects the defining relations for $W$ \cite[Theorem 3.2]{wmnds}.

Finally, we average this action applied to $f(\x)=1$ over the Weyl group to obtain the invariant function whose numerator yields the $p$-parts. First for $w \in W$, let $l(w)$ be the number of $\sigma_j$ in any reduced expression and put $sgn(w)=(-1)^{l(w)}$. Then define polynomials $j(w,\x)=sgn(w)\prod_{\alpha \in \Phi(w)}|p|^{n(\alpha)d(\alpha)}\x^{n(\alpha)\alpha}$ and $\Delta(\x)=\prod_{\alpha>0} (1-|p|^{n(\alpha)d(\alpha)}\x^{n(\alpha)\alpha})$.  The invariant function we seek is 
\[
F(\x;\ell):=\frac{1}{\Delta(x)}\sum_{w \in W} j(w,\x)(1|_{\ell}w)(\x).
\]
In particular, for $D(\x)=\prod_{\alpha>0} (1-|p|^{n(\alpha)d(\alpha)-1}\x^{n(\alpha)\alpha})$ the product $N(\x;\ell):=F(\x;\ell)D(\x)$ is polynomial in the $x_i$ \cite[Theorem 3.5]{wmnds}. For each $\ell$, the $p$-part is the polynomial $N(\x;\ell)$. For details on how the $p$-parts determine the global series coefficients, see \cite[Section 4]{wmnds}.

\begin{rem*}We think of the $p$-parts as ``metaplectic'' symmetric functions. Indeed when $n=1$, we recover a deformation of the Weyl character formula from $N(\x;\ell)=F(\x;\ell)D(\x)$. \end{rem*}

\section{The support of $N(\x;\ell)$}
\label{sec:rec}
The main goal of this section is to describe the support of the $p$-parts $N(\x;\ell)$. The key tool will be an explicit set of recurrence relations on the coefficients of $N(\x;\ell)$.
\subsection{A recurrence relation}
The invariance of $F(\x;\ell)$ under the $W$-action \eqref{eqn:act} induces a recurrence relation on the coefficients of $N(\x;\ell)$. This recurrence relation is summarized in the following theorem, which extends \cite[Theorem 3.1]{ppart} from $\Phi$ simply laced and $n=2$ to all $\Phi$ and $n$.
\begin{thm}
\label{thm:rec}
Let $N(\x;\ell)=\sum a_\lambda \x^\lambda$ and fix a simple reflection $\sigma=\sigma_k\in W$. Let $\mu=\sigma_k\bullet\lambda$, $\alpha=\alpha_k$, $m=n(\alpha_k)$, $\delta=\delta_k(\lambda)$, and $\nu=m-(\delta)_m$. Then, if $\delta\equiv 0\pmod m$, we have
\begin{align}
\label{eqn:rec1}
-|p|^{1+m}a_{\lambda-m\alpha}+a_\lambda&=-|p|^{1-\delta}\left(a_\mu-|p|^{-(1+m)}a_{\mu+m\alpha}\right).\\
\intertext{Otherwise, }
\label{eqn:rec2}
{g^*_{-\|\alpha\|^2\delta}(p)}|p|^{1+\nu}a_{\lambda-\nu\alpha}+a_\lambda&={g^*_{-\|\alpha\|^2\delta}(p)}|p|^{1-\delta}\left(a_\mu+(g^*_{-\|\alpha\|^2\delta}(p))^{-1}|p|^{-(1+\nu)}a_{\mu+\nu\alpha}\right).
\end{align}
\end{thm}

\begin{proof}
Put $l=l_k$. Let $f \in \tilde{A}$ and $g(\x) \in \tilde{A}_\beta$, with $\beta \in \Lambda'$. For all $w \in W$, we have \cite[Lemma 3.4]{wmnds}  
\begin{equation}
\label{eqn:actprod}
(fg|_\ell w)(\x)=g(w\x)(f|_\ell w)(\x).
 \end{equation}  
It follows from $D(\x) \in \tilde{A}_0$ that \((F|_\ell \sigma)(\x;\ell)=(N |_\ell\sigma)(\x;\ell)/D(\sigma \x)\). The $W$-invariance of $F(\x;\ell)$ yields 
\begin{equation}
\label{eqn:recact}
N(\x;\ell)=\frac{D(\x)}{D(\sigma\x)}(N|_\ell\sigma)(\x;\ell)
\end{equation}
and one checks that
\begin{equation}
\label{eqn:dsigma}
\frac{D(\x)}{D(\sigma\x)}=\frac{|p|^{m+1}\x^{m\alpha}(1-|p|^{m-1}\x^{m\alpha})}{|p|^{m+1}\x^{m\alpha}-1}.
\end{equation}

For $\lambda \in \Lambda$, we calculate the $\x^\lambda$ coefficient in \eqref{eqn:act}. To isolate the terms on the right-hand side that contribute, we define new functions
\begin{align*}
P_{\lambda}(x)&=(1-1/|p|)(|p| x)^{l+1-(\delta(\lambda))_m}\\
Q_{\lambda}(x)&=-g^*_{-\|\alpha\|^2\delta(\lambda)}(p)(|p| x)^{l+1-m}\\
R_\lambda(x)&=g^*_{-\|\alpha\|^2\delta(\lambda)}(p)(|p| x)^{l+1}.
\end{align*}
Substituting \eqref{eqn:dsigma} into \eqref{eqn:recact} using \eqref{eqn:act}, we obtain 
\begin{equation}
\label{eqn:termrec}
\sum (a_{\lambda-m\alpha}|p|^{m+1}-a_\lambda)\x^\lambda=\sum a_\lambda\left[P_\lambda(x)+Q_\lambda(x)+R_\lambda(x)\right]|p|^{m-l+\delta(\lambda)}\x^{m\alpha+\sigma\lambda},
\end{equation}
where we have used that the change of variables action under $\sigma$ takes $\x^\lambda$ to $|p|^{d(\sigma\lambda-\lambda)}\x^{\sigma\lambda}$ and that $d(\sigma\lambda-\lambda)=\delta(\lambda)-l-1$.

A straightforward calculation shows that the terms on the right-hand side that contribute to the $\x^\lambda$ coefficient are 
\begin{align*}
a_\gamma P_\gamma(x)|p|^{m-l+\delta_k(\gamma)}\x^{m\alpha+\sigma\gamma}, \mbox{ with } \gamma&=\sigma \bullet \lambda +\nu\alpha\\
a_\gamma Q_\gamma(x)|p|^{m-l+\delta_k(\gamma)}\x^{m\alpha+\sigma\gamma}, \mbox{ with } \gamma&=\sigma \bullet \lambda\\
a_\gamma R_\gamma(x)|p|^{m-l+\delta_k(\gamma)}\x^{m\alpha+\sigma\gamma}, \mbox{ with } \gamma&=\sigma \bullet \lambda + m\alpha.
\end{align*}

For convenience, we show the computation for the $P_\gamma$ term in the case $(\delta(\lambda))_m\neq0$. For any $\gamma\in \Lambda$, the monomial contribution from $P_\gamma(x)\x^{m\alpha+\sigma\gamma}$ is $\x^{\sigma \bullet \gamma+(m-(\delta(\gamma))_m)\alpha}$. We need only check that the exponent is $\lambda$ when $\gamma=\sigma \bullet \lambda + \nu\alpha$. This requires simplifying
\(
[\sigma\bullet(\sigma\bullet\lambda+\nu\alpha)]+[(m-(\delta(\sigma\bullet\lambda+\nu\alpha))_m) \alpha]
\).
The first term is $\lambda-\nu\alpha$. Thus it suffices to show that the second term is $\nu\alpha$ or equivalently that
\(
(\delta(\sigma\bullet\lambda+\nu\alpha))_m=(\delta(\lambda))_m.
\)
We have $\nu=(\delta(\sigma\bullet\lambda))_m$. It follows that \(\sigma\bullet\lambda+\nu\alpha=\sigma\bullet\lambda+(\delta(\sigma\bullet\lambda))_m\alpha\), and
\begin{align*}
(\delta(\sigma\bullet\lambda+(\delta(\sigma\bullet\lambda))_m\alpha))_m
&=(d(\lambda-\sigma\bullet\lambda-2(\delta(\sigma\bullet\lambda))_m\alpha))_m\\
&=(\delta(\sigma\bullet\lambda)-2(\delta(\sigma\bullet\lambda))_m)_m\\
&=(-\delta(\lambda)-2(m-(\delta(\lambda))_m))_m\\
&=(\delta(\lambda))_m.
\end{align*}

Checking the contributions for $Q_\gamma$ and $R_\gamma$ is similar. For the $Q_\gamma$ term, one must show that $l+1+\sigma\gamma=\lambda$ when $\gamma=\sigma\bullet\lambda$. For the $R_\gamma$ term, we need that $(m+l+1)\alpha+\sigma\gamma=\lambda$ when $\gamma=\sigma\bullet\lambda+m\alpha$.  These statements are both clear from definition \eqref{eqn:bullet}.

Collecting the $\x^\lambda$ coefficients and moving all terms of \eqref{eqn:termrec} to the right-hand side, we obtain a five term recurrence relation:
\begin{equation}
\label{eqn:5rec1}
\begin{array}{lll}
0&=&a_\lambda -|p|^{m+1}a_{\lambda-m\alpha}-(1-1/|p|)|p|^{1+m+\delta(\gamma)-\delta(\gamma)_m}a_{\mu+\nu\alpha}\\
&&+\,g^*_{-\|\alpha\|^2\delta(\lambda)}(p)|p|^{1-\delta(\lambda)-m}a_{\mu+m\alpha}-g^*_{-\|\alpha\|^2\delta(\lambda)}(p)|p|^{1-\delta(\lambda)}a_\mu,
\end{array}
\end{equation}
where in \eqref{eqn:5rec1} we put $\gamma=\mu+\nu\alpha$. We note that \eqref{eqn:5rec1} is a generalization of a relation in \cite{mohler}, which applied to the case $\Phi=A_r$ and $n \gg r$.

We next apply \eqref{eqn:termrec} a second time, now with $\x^{\mu+m\alpha}$ as the monomial on the left-hand side. First, we calculate the contributions to the $\x^{\mu+m\alpha}$ coefficient on the right-hand side. We have
\begin{align*}
a_\gamma P_\gamma(x)|p|^{m-l+\delta_k(\gamma)}\x^{m\alpha+\sigma\gamma}, \mbox{ with }\gamma&=\lambda - \nu\alpha\\
a_\gamma Q_\gamma(x)|p|^{m-l+\delta_k(\gamma)}\x^{m\alpha+\sigma\gamma}, \mbox{ with }\gamma&=\lambda-m\alpha\\
a_\gamma R_\gamma(x)|p|^{m-l+\delta_k(\gamma)}\x^{m\alpha+\sigma\gamma}, \mbox{ with }\gamma&=\lambda.
\end{align*}
Again, when we collect the coefficients, now moving all terms to the left-hand side, we obtain a second five-term recurrence relation: 

\begin{equation}
\label{eqn:5rec2}
\begin{array}{rrr}
g^*_{-\|\alpha\|^2\delta(\lambda)}(p)|p|^{2-\delta(\lambda)}a_\mu-g^*_{-\|\alpha\|^2\delta(\lambda)}(p)|p|^{1-\delta(\lambda)-m}a_{\mu+m\alpha}+|p|^{1+m}a_{\lambda- m\alpha}&&\\
-|p|a_\lambda-(1-1/|p|)g^*_{-\|\alpha\|^2\delta(\lambda)}(p)|p|^{2-\delta(\lambda)+\delta(\gamma)-\delta(\gamma)_m}a_{\lambda-\nu\alpha}&=&0,
\end{array}
\end{equation}
where in \eqref{eqn:5rec2} we put $\gamma=\lambda-\nu\alpha$. Note that in comparison with \eqref{eqn:5rec1}, we have multiplied each term by $g^*_{-\|\alpha\|^2\delta(\lambda)}(p)|p|^{1-\delta(\lambda)-m}$. 

Adding \eqref{eqn:5rec1} and \eqref{eqn:5rec2} and simplifying, we obtain the $m$ recurrence relations \eqref{eqn:rec1} and \eqref{eqn:rec2} stated in the theorem. For $(\delta(\lambda))_m \neq 0$, we note that when $\gamma=\lambda-\nu\alpha$, we have
\(
1-\delta(\lambda)+\delta(\gamma)-(\delta(\gamma))_m=1+m-(\delta(\lambda))_m
.\)
Similarly,  when $\gamma=\mu+\nu\alpha$, we have
\(
m+\delta(\gamma)-(\delta(\gamma))_m=(\delta(\lambda))_m-\delta(\lambda)-m.
\)
\end{proof}

\subsection{Support of $p$-parts}
We are now able to describe the support of $N(\x;\ell)$. In particular, we show that $N(\x;\ell)$ is supported on a shifted weight polytope for a lowest weight representation. 

Recall that $\theta=\sum_{i=1}^r (l_i+1)\varpi_i$.  Let $\Pi_\theta$ be the convex hull of the points $\theta-w\theta$ for $w \in W$. More precisely, $\Pi_\theta$ is the weight polytope for the irreducible representation of {\em lowest} weight $-\theta$, shifted by $\theta$. If $\Theta$ is the set of dominant weights in the representation of {\em highest} weight $\theta$, then all points of $\Pi_\theta$ have the form $\theta-w\xi$ where $w \in W$ and $\xi \in \Theta$. 

Our result is that the only nonzero coefficients of $N(\x;\ell)$ are those associated to the points in $\Pi_\theta$. The following theorem extends \cite[Theorem 3.2]{ppart} to all $\Phi$ and all $n$.

\begin{thm}
\label{thm:support}
The support of $N(\x;\ell)$ is contained in $\Pi_\theta$.
\end{thm} 
Figures \ref{fig:a2supp00} and \ref{fig:b2supp00} illustrate Theorem \ref{thm:support} for $n=3$ with $\Phi=A_2$ and $n=2$ with $\Phi=B_2$. The plotted points represent the nonzero coefficients of $N(\x;\ell)$, where $x_1^{k_1}x_2^{k_2}$ corresponds to $\x^{k_1\alpha_1+k_2\alpha_2}$.
\begin{figure}[!ht]
\centering
\begin{subfigure}{.49\linewidth}\centering
\resizebox{\linewidth}{!}{\input{a2supp00.tikz}}
\caption{$A_2$, $f(\x;0,0)$}

\end{subfigure}
\begin{subfigure}{.49\linewidth}\centering
\resizebox{\linewidth}{!}{\input{a2supp11.tikz}}
\caption{$A_2$, $f(\x;1,1)$}
\end{subfigure}
\caption{}
\label{fig:a2supp00}
\end{figure}

\begin{figure}[!ht]

\centering
\begin{subfigure}{.49\linewidth}\centering
\resizebox{\linewidth}{!}{\input{b2supp00.tikz}}
\caption{$B_2$, $f(\x;0,0)$}
\end{subfigure}
\begin{subfigure}{.49\linewidth}\centering
\resizebox{\linewidth}{!}{\input{b2supp24copy.tikz}}
\caption{$B_2$, $f(\x;2,4)$}
\end{subfigure}
\caption{}
\label{fig:b2supp00}
\end{figure}

We argue as in \cite{ppart} using the recurrence relation of Theorem \ref{thm:rec}. We require two geometric lemmas. Note that $\Pi_\theta$  is cut out by the inequalities
\begin{equation}
\label{eqn:ineq}
( w \varpi_i,\x-(\theta-w\theta))\geq 0\hspace{1cm} \, w \in W,\, i=1,\ldots,r.
\end{equation}
This system of inequalities is redundant in the sense of the following lemma.
\begin{lem}{\cite[Lemma 3.4]{ppart}}
\label{lem:ineq}
Let $\sigma_k \in \mathcal{R}(w):=\{\sigma_i:l(w\sigma_i)<l(w)\}$ and let $u=w\sigma_k$. If $k \neq j$, the inequalities 
\[
( w\varpi_j,\x-(\theta-w\theta))\geq 0
\]
and
\[
( u\varpi_j,\x-(\theta-u\theta)) \geq 0
\]
are equivalent. \qed
\end{lem}
\noindent Moreover, $\Pi_\theta$ has bounded support.
\begin{lem}{\cite[Lemma 3.5]{ppart}}
\label{lem:ray}
Let $\mu=\theta-w\theta$ be a vertex of $\Pi_\theta$, and suppose $\sigma_k \in \mathcal{L}(w):=\{\sigma_i:l(\sigma_iw)<l(w)\}$. Then any lattice point of the form $\mu+b\alpha_k$, where $b$ is a positive integer, lies outside of $\Pi_\theta$. Similarly, let $u=\sigma_kw$ and let $\lambda=\theta-u\theta$. Then any point of the form $\lambda-b\alpha_k$, where $b$ is a positive integer, lies outside of $\Pi_\theta$. \qed
\end{lem}

\begin{proof}[Proof of Theorem \ref{thm:support}]
To avoid repetition, we summarize the key ideas here and refer the reader to \cite{ppart} for details. Recall that $N(\x;\ell)=\sum a_\lambda \x^\lambda$ is a polynomial. The goal is to show that  for $\lambda \notin \Pi_\theta$ we have $a_\lambda=0$. Equivalently we show that if
$(w \varpi_i,\lambda-(\theta-w\theta))<0$
for some $w \in W$ and $i=1,\ldots, r$, then $a_\lambda=0$.
We induct on the length of $w$. If $l(w)=0$ and $\lambda$ violates the inequalities active at the origin, then clearly $a_\lambda=0$. Otherwise we would have polar terms, a contradiction since $N(\x;\ell)$ is polynomial. 

Now suppose that $l(w)>0$ and that we have verified the inequalities at all vertices where $\theta-u\theta$, where $l(u)<l(w)$. It follows from Lemma 3.3 and the induction hypothesis that it suffices  to consider the case when $\mathcal{R}(w)=\{\sigma_k\}$ for some $k=1,\ldots, r$ and to prove $a_\lambda=0$ if $\lambda$ violates the inequality 
\begin{equation}
\label{eqn:sigk}
(w\varpi_k,\x-(\theta-w\theta))\geq0.
\end{equation}

For the sake of contradiction, let $\sigma_j \in \mathcal{L}(w)$ and choose $\mu \in \Lambda$ such that $\mu$ violates \eqref{eqn:sigk}, but $a_\mu \neq 0$. Further assume that $a_{\mu'}=0$ for all $\mu'=\mu+b\alpha_j$ with $b>0$.  
In other words, $\mu$ is the final point of support on the ray $\mu+b\alpha_j$. By Lemma \ref{lem:ray}, such a $\mu$ exists.

We apply Theorem \ref{thm:rec} with $\sigma=\sigma_j$ to $a_\mu$, where $a_\mu$ is the first coefficient on the right-hand side of the relation. Then $\mu=\sigma\bullet\lambda$ for some $\lambda \in \Lambda$. Here we set $m=n(\alpha_j)$ and $\delta=\delta_j(\lambda)$. We have $a_{\mu+b\alpha_j}=0$ for all $b>0$, and if $\delta \equiv 0 \mod m$ then $g_{-\|\alpha\|^2\delta}^\ast(p)=-1$. Thus, the right-hand side yields $g^\ast_{-\|\alpha\|^2\delta}(p)|p|^{1-\delta}a_\mu$ regardless of the value of $\delta \mod m$.

In the case $\delta\equiv 0\mod m$, the left-hand side is
\[
-|p|^{1+m}a_{\lambda-m\alpha_j}+a_\lambda.
\]
Otherwise, we have 
\[
{g^*_{-\|\alpha\|^2\delta}(p)}|p|^{1+m-(\delta)_m}a_{\lambda-(m-(\delta)_m)\alpha_j}+a_\lambda.
\]
One checks that $a_\lambda$ vanishes by the induction hypothesis, since $l(\sigma_jw)<l(w)$ implies $\lambda=\theta-\sigma_jw\theta$ violates the inequalities active at $\theta-\sigma_jw\theta$. Again by Lemma \ref{lem:ray} this implies $a_{\lambda-m\alpha}$ and $a_{\lambda-(m-(\delta)_m)\alpha_j}$ also vanish. Therefore, the left-hand side is identically zero and $a_\mu$ vanishes. It follows that $a_\mu=0$ unless $\mu$ satisfies \eqref{eqn:sigk}.
\end{proof}

\subsection{Support of modified $p$-parts}
We now take a slight detour from our main goal to consider the support of the modified $p$-parts $f(\x;\ell):=\Delta(\x)F(\x;\ell)=\Delta(\x)N(\x;\ell)/D(\x)$. As mentioned in the introduction, the $f(\x;\ell)$ coincide, up to a change of variables, with Weyl group multiple Dirichlet series associated to the rational function field over a finite field. We expect that Weyl group multiple Dirichlet series associated to function fields over finite fields encode arithmetic data about the arithmetic of the defining curve, much like zeta functions. Theorem \ref{thm:gap} has applications toward understanding the support of these series, which the author will explore in a forthcoming paper. We include the result here as a complement to Theorem \ref{thm:support}.

\begin{thm}
\label{thm:gap}
The support of $f(\x;\ell)$ lies outside the interior of the polytope $\Pi_\theta$.
\end{thm}
Figures \ref{fig:a2gap00} and \ref{fig:b2gap00} illustrate Theorem \ref{thm:gap} for $n=3$ with $\Phi=A_2$ and $n=2$ with $\Phi=B_2$. The plotted points represent the nonzero coefficients of $f(\x;\ell)=\Delta(\x)F(x;\ell)$, where $x_1^{k_1}x_2^{k_2}$ corresponds to $\x^{k_1\alpha_1+k_2\alpha_2}$.

\begin{figure}[!ht]
\centering
\begin{subfigure}{.49\linewidth}\centering
\resizebox{\linewidth}{!}{\input{a2gap00.tikz}}
\caption{$A_2$, $f(\x;0,0)$}
\label{fig:a2gap00}
\end{subfigure}
\begin{subfigure}{.49\linewidth}\centering
\resizebox{\linewidth}{!}{\input{a2gap11.tikz}}
\caption{$A_2$, $f(\x;1,1)$}
\end{subfigure}
\caption{}
\label{fig:a2gap00}
\end{figure}

\begin{figure}[!ht]

\centering
\begin{subfigure}{.49\linewidth}\centering
\resizebox{\linewidth}{!}{\input{b2gap00.tikz}}
\caption{$B_2$, $f(\x;0,0)$}
\end{subfigure}
\begin{subfigure}{.49\linewidth}\centering
\resizebox{\linewidth}{!}{\input{b2gap24.tikz}}
\caption{$B_2$, $f(\x;2,4)$}
\end{subfigure}
\caption{}
\label{fig:b2gap00}
\end{figure}

Before giving a proof, we recall the following. Let $\Phi(w)=\{\alpha \in \Phi^+:w\alpha \in \Phi^{-}\}$ be the set of positive roots made negative by $w$. One can show that $\#\Phi(w)=l(w)$.  If $l(\sigma_kw)=l(w)+1$, then \cite[Section 5.6]{refgroups}
\begin{equation*}
\Phi(\sigma_kw)=\Phi(w)\cup\{w^{-1}\alpha_k\}.
\end{equation*}
In particular, if $l(w\sigma_k)=l(w)+1$, then
\begin{equation}
\label{eqn:phiw2}
\Phi(\sigma_kw^{-1})=\Phi(w^{-1})\cup\{w\alpha_k\}.
\end{equation}

\begin{proof} By definition $\Delta(\x)F(\x;\ell)=\sum_{w \in W} j(w,\x)(1|_\ell w)(\x)$. Thus, it is enough to show that 
the support of $j(w,\x)(1|_\ell w)(\x)$ is contained in the cone determined by $\Phi(w)$ and shifted by $\theta-w\theta$. 

We prove this claim by induction on $l(w)$. When $l(w)=0$, the statement is clear. Now assume that $l(w\sigma_k)=l(w)+1$. We wish to show that 
\(
j(w\sigma_k,\x)(1|_\ell w\sigma_k)(\x)
\)
is supported on the cone determined by $\Phi(w\sigma_k)$ and shifted by $\theta-\sigma_kw\theta$. By our assumption on the relative lengths of $w$ and $w\sigma_k$, we have
\(
\Phi(w\sigma_k)=\sigma_k(\Phi(w))\cup\{\alpha_k\}.
\)

Since $j(w,\x) \in \tilde{A}_0$ for all $w$, it follows from \eqref{eqn:actprod} that 
\(
[j(w,\x)f|_\ell \sigma_k](\x)=j(w,\sigma_k\x)[f|_\ell(\x)].
\)
By \cite[Lemma 3.3]{wmnds}, up to sign we have
\begin{align*}
j(w,\sigma_k\x)&=\prod_{\alpha \in \Phi(w)} |p|^{n(\alpha)d(\alpha)}(\sigma_k\x)^{n(\alpha)\alpha}\\
&=\prod_{\alpha \in \sigma_k(\Phi(w))}|p|^{n(\alpha)d(\alpha)}\x^{n(\alpha)\alpha}\\
&=\prod_{\alpha \in \Phi(w\sigma_k)\backslash\{\alpha_k\}}|p|^{n(\alpha)d(\alpha)}\x^{n(\alpha)\alpha}\\
&=j(w\sigma_k,\x)/|p|^{n(\alpha_k)d(\alpha_k)}\x^{n(\alpha_k)\alpha_k}
\end{align*}
where we have used that $n(\alpha)=n(\sigma_k\alpha)$. From the above equality, we see
\begin{align*}
j(w\sigma_k,\x)(1|_lw\sigma_k)(\x)&=j(w\sigma_k,\x)[(1|_\ell w)|_l\sigma_k](\x)\\
&=|p|^{n(\alpha_k)d(\alpha_k)}\x^{n(\alpha_k)\alpha_k}\left(j(w,\sigma_k\x)[(1|_\ell w)|_\ell\sigma_k](\x)\right)\\
&=|p|^{n(\alpha_k)d(\alpha_k)}\x^{n(\alpha_k)\alpha_k}\left([j(w,\x)(1|_\ell w)]|_\ell \sigma_k\right)(\x).
\end{align*}

Our inductive hypothesis implies that $j(w,\x)(1|_\ell w)(\x)$ is supported on the cone defined by $\Phi(w)$ and shifted by $\theta-w\theta$. Thus, we can assume that $$j(w,\x)(1|_\ell w)(\x)=\sum f_\beta(\x),$$ where each $f_\beta(\x)$ is a monomial supported on $\Phi(w)$ shifted by $\theta-w\theta$.
Using \eqref{eqn:act}, we have 
\begin{align*}
\left([j(w,\x)(1|_\ell w)]|_\ell \sigma_k\right)(\x)&=\sum [P_{\beta,\ell,k}(x_k)+Q_{\sigma_k\bullet \beta,\ell,k}(x_k)]f_\beta(\sigma_k\x),
\end{align*}
where up to constants $a_\beta$, 
\(
f_\beta(\sigma_k\x)=a_\beta |p|^{d(\sigma_k\beta-\beta)}\x^{\sigma_k\beta}.
\)
By definition, we see that 
\begin{align*}
P_{\beta,\ell,k}(x_k)f_\beta(\sigma_k\x)&=\x^{\sigma_k \bullet \beta} \x^{-\delta_{k}(\beta)_{n(\alpha_k)}} \tilde{P}(x_k)\\
\intertext{and}
Q_{\sigma_k\bullet \beta,\ell,k}(x_k)f_\beta(\sigma_k\x)&=\x^{\sigma_k\bullet\beta}\x^{-n(\alpha_k)} \tilde{Q}(x_k),
\end{align*}
where $\tilde{P}(x_k)$ and $\tilde{Q}(x_k)$ are rational functions supported on the ray determined by $\alpha_k$. After multiplying by $x_k^{n(\alpha_k)}$, each of these two terms is supported on $\sigma_k(\Phi(w))\cup \alpha_k=\Phi(w\sigma_k)$. Moreover, since the original cone had been shifted to the vertex $\x^{w \bullet 0}$, which corresponds to $w$, the new cone is shifted to the vertex $\x^{\sigma_k\bullet w\bullet 0}$, which corresponds to a reflection under $\sigma_k$.
\end{proof}

\section{Stable coefficients}
\label{sec:stable}
In this section we compare the methods of \cite{wmds2, wmds4} and \cite{wmnds} to define $p$-parts of Weyl group multiple Dirichlet series. More specifically, we compare these methods in the ``stable'' case $n \gg r$. In this case, the only nonzero coefficients of $N(\x;\ell)$ are those attached to the vertices of the polytope $\Pi_\theta$ \cite{wmds4,wmds2}. We call these coefficients the stable coefficients. For the precise stability condition, see \cite[Equation 20]{wmds4}. When $\Phi$ is type $A$, it is enough to have $n\geq \sum_{i=1}^r (l_i+1)$.

As in \cite{ppart}, we caution the reader that we now make a slight change in notation: the element $w$ now corresponds to the coefficient of $\lambda=\theta-w^{-1}\theta$. This change is to aid the comparison of the sets $\Phi(w^{-1})$ and $\Phi(\sigma_kw^{-1})$. For each $\theta$, \cite{wmds4,wmds2} defines the stable coefficients 
\begin{equation}
\label{eqn:stabc}
A_\lambda=\prod_{\alpha \in \Phi(w^{-1})} g_{\|\alpha\|^2}(p^{d_\theta(\alpha)-1},p^{d_\theta(\alpha)}),
\end{equation}
where $\lambda=\theta-w\theta$ for some $w \in W$ and $d_\theta(\alpha):=\lang \theta,\alpha\rang$.

\begin{thm}
\label{thm:stable}
Let $\theta-w^{-1}\theta=\sum k_i\alpha_i$, and let $N(\x;\ell)=\sum a_\lambda \x^\lambda$ be the $p$-part constructed via \cite{wmnds}. Then  
\(
a_\lambda=A_\lambda.
\)
In other words, the stable coefficients of \cite{wmds2,wmds4} and \cite{wmnds} agree. 
\end{thm}

\begin{proof}
We follow \cite{ppart} and induct on $l(w)$. Assume that $a_0=1$. If $l(w)=0$, we have an empty product on the right hand side of \eqref{eqn:stabc} and the statement holds trivially. Suppose that $a_\lambda=A_\lambda$ on all $v \in W$ with $l(v) \leq l(w)$. Let $\mu=\theta-\sigma_jw^{-1}\theta$, with $l(\sigma_jw^{-1})=l(w^{-1})+1$. In other words, we assume $\mu=\sigma_j\bullet \lambda$ and $\mu<\lambda$. Applying \eqref{eqn:rec1} or \eqref{eqn:rec2} with $a_\mu$ on the right-hand side, the outer terms vanish by Lemma \ref{lem:ray}. We have \(a_\mu=-|p|^{\delta(\lambda)-1}a_\lambda\) if $\delta(\lambda) \equiv 0 \mod n(\alpha_j)$ and \(a_\mu=(g^\ast_{-\|\alpha_j\|^2\delta(\lambda)}(p))^{-1}|p|^{\delta(\lambda)-1}a_\lambda.\) otherwise. On the other hand, it follows from \eqref{eqn:phiw2} and our induction hypothesis that
\begin{align*}A_\mu&=\prod_{\alpha \in \Phi(\sigma_jw^{-1})} g_{\|\alpha\|^2}(p^{d_\theta(\alpha)-1},p^{d_\theta(\alpha)})\\
&=\prod_{\alpha \in \Phi(w^{-1})}g_{\|\alpha\|^2}(p^{d_\theta(\alpha)-1},p^{d_\theta(\alpha)})\times g_{\|w\alpha_j\|^2}(p^{d_\theta(w\alpha_j)-1},p^{d_\theta(w\alpha_j)})\\
&=a_\lambda \times g_{\|w\alpha_j\|^2}(p^{d_\theta(w\alpha_j)-1},p^{d_\theta(w\alpha_j)})
\end{align*}
It remains to show 
\[
g_{\|w\alpha_j\|^2}(p^{d_\theta(w\alpha_j)-1},p^{d_\theta(w\alpha_j)})=\left\{\begin{array}{cl} -|p|^{\delta(\lambda)-1}&\mbox{ if $\delta(\lambda) \equiv 0 \mod n(\alpha_j)$}\\
(g^\ast_{-\|\alpha_j\|^2\delta(\lambda)}(p))^{-1}|p|^{\delta(\lambda)-1} &\mbox{ otherwise.}\end{array}\right.
\]

When $\delta(\lambda) \equiv 0 \mod {n(\alpha_j)}$, we have $g^\ast_{-\|\alpha_j\|^2\delta(\lambda)}(p)=-1$. Otherwise, we use \eqref{eqn:gauss2} and \eqref{gdef} to rewrite \((g^\ast_{-\|\alpha_j\|^2\delta(\lambda)}(p))^{-1}|p|^{\delta(\lambda)-1}=g_{\|\alpha_j\|\delta(\lambda)}(p)|p|^{\delta(\lambda)-1}.\) From \eqref{eqn:gauss}, it suffices to show that $\delta_j(\lambda)=d_\theta(w\alpha_j)$.
By definition
\begin{align*}
\delta(\lambda)
&=d(\mu-\lambda)=d\left(\left(\sum_i k_i c(i,j)\right)\alpha_j-(l_j+1)\alpha_j\right)
=\sum k_i c(i,j)-l_j-1.
\end{align*}
Also
\begin{align*}
d_\theta(w\alpha_i)&=\lang \theta, w\alpha_j\rang\\
&=\lang \sigma_jw^{-1}\theta,-\alpha_j\rang\\
&=\lang \theta-\sum k_i\alpha_i,-\alpha_j\rang\\
&=-l_j-1+\lang\sum k_i\alpha_i,\alpha_j\rang\\
&=\sum k_i c(i,j)-l_j-1.
\end{align*}
This implies the result.
\end{proof}
We remark that in the stable case $n \gg r$, the only coefficients of $N(\x;\ell)$ are those associated to the vertices of $\Pi_\theta$. Thus Theorem \ref{thm:stable} shows that in the stable case the $p$-parts of \cite{wmds2,wmds4} and \cite{wmnds} agree.

\section{Unstable coefficients}
\label{sec:unstable}
In this section we address the extent to which Theorem \ref{thm:rec} determines $N(\x;\ell)=\sum a_\lambda\x^\lambda$. Specifically, in Theorem \ref{thm:dim} we find the dimension of space of all polynomials satisfying the recurrence relations of Theorem \ref{thm:rec}. Throughout this section, for $\ell=(l_1,\ldots,l_r) \in (\Z_{\geq 0})^r$ we let $\theta=\sum_{i=1}^r (l_i+1)\varpi_i$. Recall that we have defined $\Theta$ to be the set of all dominant weights in the irreducible representation of highest weight $\theta$. We now let $\Theta^+$ be subset of {\em regular} dominant weights. 

The following theorem is a generalization of \cite[Theorem 5.7]{ppart} and gives an upper bound on the dimension of the space of polynomials satisfying Theorem \ref{thm:rec}.
\begin{thm}
\label{thm:unstab}
Suppose that the coefficients $a_{\theta-\xi}$ of $N(\x;\ell)$ are known for all $\xi \in \Theta^+$. Then $N(\x;\ell)$ is completely determined by the relations of Theorem \ref{thm:rec} after setting $a_0=1$. 
\end{thm}

Our argument follows \cite{ppart}. We consider $w$-orbits of the coefficients. Recall that by Theorem \ref{thm:support}, the support of $N(\x;\ell)=\sum a_\lambda \x^\lambda$ is contained among all $\lambda=\theta-w\xi$ such that $w \in W$ and $\xi \in \Theta$. For any $\xi \in \Theta$, define $O_\xi:=\{\theta-w\xi:w \in W\}$ as the $W$-orbit of the coefficient $a_{\theta-\xi}$ under the $\bullet$ action. Let $\mathcal{O}=\{O_\xi:\xi \in \Theta\}$ be the set of all such orbits. There is a natural partial order on $\mathcal{O}$ given by the poset relation on the weights: we say $O_\xi \leq O_{\xi'}$ if and only if $\xi \preceq \xi'$. Under this identification of $\xi$  with $\theta-\xi$, the condition $\xi \preceq \xi'$ becomes $(\theta-\xi) \succeq (\theta-\xi')$. 

\begin{proof}
We induct on the poset $\cO$. First, we know all the points in the orbit $O_\theta$ by Theorem \ref{thm:stable}. Now fix $\xi \in \Theta$ such that $\xi \neq \theta$, and assume that we have determined the coefficients attached to all orbits $O$ with $O>O_\theta$. We consider two cases: $\xi$ regular or not. 

Assume that $\xi$ is regular, and let $\lambda=\theta-\xi$. By assumption, the value of the coefficient $a_\lambda$ associated with $\xi$ is known. We must show that we can determine $a_\delta$ for all $\delta \in O_\xi$. Any such $\delta$ is of the form $\theta-w\xi=w \bullet \lambda$ for $w \in W$ and thus can be obtained from successive application of simple reflections $\sigma_j$. From relation \eqref{eqn:rec1} or \eqref{eqn:rec2}, when we apply $\sigma_j$ to $a_\lambda$, the outer term on the left-hand side is either $\theta-(\xi+\nu\alpha_j)$ or $\theta-(\xi+n\alpha_j)$, where $m=n(\alpha_j)$ and $\nu=m-(\delta(\lambda))_m$. In both cases, \cite[Lemma 5.3]{ppart} implies that these terms come from a previously determined orbit $O_{\xi'}$ with $O_{\xi'}>O_{\xi}$. 
Under the action $\sigma_j\bullet(\theta-\xi')$ we obtain the outer term on the right-hand side, so this term also belongs to $O_{\xi'}$ and hence is previously determined. Since we know three out of four terms of the relation, $a_\delta$ is determined.
 
Now, assume that $\xi$ is not regular. Since $\#O_\xi<\#W$, there exists a simple reflection $\sigma_j$ such that $a_\lambda$ is taken to itself under relation \eqref{eqn:rec1} or \eqref{eqn:rec2}. By \cite[Lemma 5.5]{ppart}, all other $a_{\lambda'}$ involved in the recurrence are predetermined; therefore we know $a_\lambda$. We may now successively apply \eqref{eqn:rec1} or \eqref{eqn:rec2} to determine the remaining coefficients $\theta-w\xi$ in the orbit $O_\xi$.
\end{proof}

\begin{cor}
\label{cor:unstab}
When $\ell=(0,\ldots, 0)$, $N(\x;\ell)$ is completely determined by the relations from Theorem \ref{thm:rec} after setting the constant term to 1.
\end{cor}

\begin{proof}
See \cite[Corollary 5.8]{ppart}.
\end{proof}

We now show that any polynomial satisfying the recurrence relations of Theorem \ref{thm:rec} can be written as the weighted sum of $\C$-linearly independent shifted $p$-parts. The key to this result is the following lemma, which allows us to shift action \eqref{eqn:act}.

\begin{lem}[\cite{chintaoffen,mcnam2}]
\label{lem:shiftact}
Let $f(\x)\in \C[\Lambda]$ be a rational function. Write the $|_\ell$ action described in \eqref{eqn:act} as $|_\theta$. Then for $\xi \in \Theta^+$, we have \[(f|_\xi w)(\x)=(f\x^{\theta-\xi}|_\theta w)(\x)\x^{\xi-\theta}.\] Consequently, $f(\x)$ satisfies the recurrence relations of Theorem \ref{thm:rec} for $\xi$ if and only if $f(\x)\x^{\theta-\xi}$ satisfies the recurrence relations of Theorem \ref{thm:rec} for $\theta$.

\end{lem}

\begin{proof} 
For $\Phi$ type $A$, Chinta and Offen \cite{chintaoffen} define a Weyl group action on rational functions that is independent of $\theta$. Equation \cite[(9.2)]{chintaoffen} shows that this $\theta$-independent action is equivalent to action \eqref{eqn:act} after the shifting described in the statement of the lemma. In \cite{mcnam2}, McNamara generalizes the action of \cite{chintaoffen} to all $\Phi$. For the reader's convenience, we demonstrate the case when $f(\x)$ is a monomial.

Let $f_\beta(\x)=\x^\beta\in A_\beta$. 

We claim
\((f_\beta|_\xi w)(\x)\x^{\theta-\xi}=(f_\beta\x^{\theta-\xi}|_\theta w)(\x).\)
To see this, let $\sigma_k$ be a simple reflection. Writing $\xi=\sum (r_i+1)\varpi_i$, we put $(\xi)_k=b_k$. Recall that $P_{\beta,\xi,k}$ and $Q_{\beta,\xi,k}$ depend only on $\delta_{\xi,k}(\beta)$ and $r_k$. Note that for $w \in W$, \eqref{eqn:bullet} implies

\[w \bullet_{\xi} \lambda+(\theta-\xi)=w \bullet_\theta {(\lambda+\theta-\xi)}.\]

It follows that $\delta_{\theta,k}(\beta+\theta-\xi)=\delta_{\xi,k}(\beta)$. Let $C=-\sum_{i} c(i,k)$. A simple calculation shows that $(|p|x_k)^C=(f_{\beta+\theta-\xi}(\sigma_k\x))/(f_\beta(\sigma_k\x)\x^{\theta-\xi})$. Now put $\theta-\xi=\sum k_i\alpha_i$. 
Using the Cartan matrix to base change, we have $(\theta)_k-(\xi)_k=\sum k_ic(i,k)=-C$. The claim now follows from the definition of the action in \eqref{eqn:act}. \end{proof}

\begin{thm}
\label{thm:dim} The dimension of the $\C$-vector space of polynomials satisfying the recurrence relations of Theorem \ref{thm:rec} is equal to $\#\Theta^+$.
\end{thm}
\begin{proof}  Denote $N(\x;\ell)$ by $N(\x;\theta)$. Let $\mathcal{N}(\x)=\sum_{\lambda} b_\lambda \x^\lambda$ be any polynomial satisfying the recurrence relations of Theorem \ref{thm:rec}. By Theorem \ref{thm:unstab},  the coefficients $b_{\theta-\xi}$ for $\xi \in \Theta^+$ completely determine $\mathcal{N}(\x)$. Let $m_\theta=b_0$ be the constant coefficient of $\mathcal{N}(\x)$; Theorem \ref{thm:stable} shows that the stable coefficients of $m_\theta N(\x;\theta)$ and $\mathcal{N}(\x)$ agree. Write 
\[
\mathcal{N}(\x)=m_\theta N(\x;\theta)+E_\theta(\x),
\]
where $E_\theta(\x)$ is a polynomial that also satisfies the recurrence relations of Theorem \ref{thm:rec} and is supported on the orbits $O_{\xi'}=\{\x^{\theta-w\xi'}:w\in W\}$. Let $\mathcal{S}_\theta$ denote the set of corresponding $\xi'$. Since $\theta$ is the unique maximal element of $\Theta$, all such $\xi'$ satisfy $\xi'\prec\theta$.  If $S_\theta$ is empty then we have expressed $\mathcal{N}(\x)$ as a sum of shifted $p$-parts, so assume $S_\theta$ is nonempty. 

Choose $\xi\in \mathcal{S}_\theta \cap \Theta$ maximal with respect to the partial order on $L$. If $\xi$ is not regular, then there exists a simple reflection taking $\xi$ to itself. It follows again from Theorem \ref{thm:unstab} that the $\x^\xi$ coefficients of both $\mathcal{N}(\x)$ and $N(\x;\theta)$ are completely determined by the coefficients associated to orbits $O_{\xi''}$, where $\xi''$ satisfies $\xi''\succ \xi$. If the $\x^\xi$ coefficient of $E_\theta(\x)$ is zero, then so are all coefficients of $E_\theta(\x)$ in the $\xi$-orbit and we contradict $\xi \in \mathcal{S}_\theta$. Otherwise, we contradict the maximality of $\xi$. Thus, $\xi$ is regular.

Applying Lemma \ref{lem:shiftact}, $E_\theta(\x)\x^{\xi-\theta}$ satisfies \eqref{eqn:rec1} and \eqref{eqn:rec2} for the twisting parameter $\xi$. Let $m_\xi$ be the constant coefficient of $E_\theta(\x)\x^{\xi-\theta}$, and write
\[
E_\theta(\x)=m_\xi N(\x;\xi)\x^{\theta-\xi}+E_\xi(\x),
\]
where $E_\xi(\x)$ is supported on the orbits $O_{\xi'}$ with $\xi'\prec\xi$. By the same argument as above, the maximal such $\xi'$ is regular. Continuing in this fashion, after at most $\#\Theta^+$ iterations, we may write
\begin{equation}
\label{eqn:sum}
\mathcal{N}(\x)=\sum_{\xi \in \Theta^+} m_{\xi}N(\x;\xi)\x^{\theta-\xi},\hspace{1cm} m_\xi \in \C.
\end{equation} 
We note that the $N(\x;\xi)\x^{\theta-\xi}$ have disjoint support and thus are clearly $\C$-linearly independent.

\end{proof}

\subsection*{Acknowledgements}
I thank Paul E. Gunnells, who advised the thesis that led to this work for invaluable advice and guidance during this project. I also thank ICERM and the organizers and participants of the Spring 2013 semester on Automorphic Forms, Combinatorial Representation Theory and Multiple Dirichlet Series for their support during the final stages of my dissertation. In particular, thank you to Gautam Chinta, Ben Brubaker, Solomon Friedberg, Daniel Bump, Jeff Hoffstein, Anna Pusk\'{a}s, Ian Whitehead, and TingFang Lee for helpful conversations. Finally, I would like to thank the referee for helpful comments.

\end{document}

%% file: a2supp00.tikz
\begin{tikzpicture}[scale=.15,>=latex]
	        \pgfmathsetmacro\ax{3}
	        \pgfmathsetmacro\ay{0}
	        \pgfmathsetmacro\bx{3* cos(120)}
	        \pgfmathsetmacro\by{3 * sin(120)}
\begin{scope}
\clip (-15,-5) rectangle (30,30);
            	\begin{scope}[cm={\ax,\ay,\bx,\by,(0,0)}]
            	\draw[style=help lines,gray] (-15,-15) grid (15,15);
            	\draw[cm={1,0,1,1,(0,0)},style=help lines, gray] (-15,-15) grid (15,15);
            	\end{scope}
\begin{scope}[cm={\ax,\ay,\bx,\by,(0,0)},every node/.style={circle,fill=black!50!green,inner sep=0pt,minimum size=4pt}]
\node at (0,0) {};
\node at (1,0) {};
\node at (0,1) {};
\node at (2,1) {};
\node at (1,2) {};
\node at (2,2) {};
\end{scope}
\end{scope}
\draw[->,>=stealth,black] (0,0)--(\ax,\ay);
\draw[->,>=stealth,black] (0,0)--(\bx,\by);
\node[right] at  (\ax,\ay) {\(\alpha_1\)};
\node[above left] at (\bx,\by) {\(\alpha_2\)};
\end{tikzpicture}

%% file: a2supp11.tikz
\begin{tikzpicture}[scale=.15,>=latex]
	        \pgfmathsetmacro\ax{3}
	        \pgfmathsetmacro\ay{0}
	        \pgfmathsetmacro\bx{3* cos(120)}
	        \pgfmathsetmacro\by{3 * sin(120)}
\begin{scope}
\clip (-15,-5) rectangle (30,30);
            	\begin{scope}[cm={\ax,\ay,\bx,\by,(0,0)}]
            	\draw[style=help lines,gray] (-15,-15) grid (15,15);
            	\draw[cm={1,0,1,1,(0,0)},style=help lines, gray] (-15,-15) grid (15,15);
            	\end{scope}
\begin{scope}[cm={\ax,\ay,\bx,\by,(0,0)},every node/.style={circle,fill=black!50!green,inner sep=0pt,minimum size=4pt}]
\node at (0,0) {};
\node at (2,0) {};
\node at (0,2) {};
\node at (3,2) {};
\node at (2,3) {};
\node at (3,3) {};
\node at (4,2) {};
\node at (2,4) {};
\node at (4,4) {};
\end{scope}
\end{scope}
\draw[->,>=stealth,black] (0,0)--(\ax,\ay);
\draw[->,>=stealth,black] (0,0)--(\bx,\by);
\node[right] at  (\ax,\ay) {\(\alpha_1\)};
\node[above left] at (\bx,\by) {\(\alpha_2\)};
\end{tikzpicture}

%% file: b2supp00.tikz
\begin{tikzpicture}[scale=.1,>=latex]
	        \pgfmathsetmacro\ax{3 * sqrt(2)}
	        \pgfmathsetmacro\ay{0}
	        \pgfmathsetmacro\bx{3 * cos(135)}
	        \pgfmathsetmacro\by{3 * sin(135)}
\clip (-20,-5) rectangle (30,30);
            	\begin{scope}[cm={\ax,\ay,\bx,\by,(0,0)}]
            	\draw[style=help lines,gray] (-20,-20) grid (30,30);
            	\draw[cm={1,0,1,1,(0,0)},style=help lines, gray] (-20,-20) grid (30,30);
            	\draw[cm={1,1,0,1,(0,0)},style=help lines, gray] (-20,-20) grid (30,30);
            	\draw[cm={1,2,1,0,(0,0)},style=help lines, gray] (-20,-20) grid (30,30);
            	\end{scope}
\begin{scope}[cm={\ax,\ay,\bx,\by,(0,0)},every node/.style={circle,fill=blue,inner sep=0pt,minimum size=4pt}]
\node at (0,0) {};
\node at (1,0) {};
\node at (0,1) {};
\node at (1,1) {};
\node at (2,1) {};
\node at (1,3) {};
\node at (2,3) {};
\node at (3,3) {};
\node at (2,4) {};
\node at (3,4) {};
\end{scope}
\draw[->,black] (0,0)--(\ax,\ay);
\draw[->,black] (0,0)--(\bx,\by);
\node[right] at  (\ax,\ay) {\(\alpha_1\)};
\node[above left] at (\bx,\by) {\(\alpha_2\)};
\end{tikzpicture}

%% file: b2supp24copy.tikz
\begin{tikzpicture}
[scale=.08,>=latex]
\pgfmathsetmacro\ax{3 * sqrt(2)}
\pgfmathsetmacro\ay{0}
\pgfmathsetmacro\bx{3 * cos(135)}
\pgfmathsetmacro\by{3 * sin(135)}
\clip (-30,-5) rectangle (50,50);
\begin{scope}[cm={\ax,\ay,\bx,\by,(0,0)}]
\draw[style=help lines,gray] (-20,-20) grid (30,30);
\draw[cm={1,0,1,1,(0,0)},style=help lines, gray] (-20,-20) grid (30,30);
\draw[cm={1,1,0,1,(0,0)},style=help lines, gray] (-20,-20) grid (30,30);
\draw[cm={1,2,1,0,(0,0)},style=help lines, gray] (-20,-20) grid (30,30);
\end{scope}
\begin{scope}[cm={\ax,\ay,\bx,\by,(0,0)},every node/.style={circle,fill=blue,inner sep=0pt,minimum size=4pt}]
\node at (0,0) {};
\node at (0,2) {};
\node at (0,4) {};
\node at (0,5) {};
\node at (1,0) {};
\node at (1,2) {};
\node at (1,4) {};
\node at (1,5) {};
\node at (1,6) {};
\node at (1,7) {};
\node at (2,0) {};
\node at (2,2) {};
\node at (2,4) {};
\node at (2,5) {};
\node at (2,6) {};
\node at (2,7) {};
\node at (2,8) {};
\node at (2,9) {};
\node at (3,0) {};
\node at (3,2) {};
\node at (3,4) {};
\node at (3,5) {};
\node at (3,6) {};
\node at (3,7) {};
\node at (3,8) {};
\node at (3,9) {};
\node at (3,10) {};
\node at (3,11) {};
\node at (4,2) {};
\node at (4,4) {};
\node at (4,5) {};
\node at (4,6) {};
\node at (4,7) {};
\node at (4,8) {};
\node at (4,9) {};
\node at (4,10) {};
\node at (4,11) {};
\node at (5,2) {};
\node at (5,4) {};
\node at (5,5) {};
\node at (5,6) {};
\node at (5,7) {};
\node at (5,8) {};
\node at (5,9) {};
\node at (5,10) {};
\node at (5,11) {};
\node at (5,12) {};
\node at (5,13) {};
\node at (6,4) {};
\node at (6,5) {};
\node at (6,6) {};
\node at (6,7) {};
\node at (6,8) {};
\node at (6,9) {};
\node at (6,10) {};
\node at (6,11) {};
\node at (6,12) {};
\node at (6,13) {};
\node at (7,4) {};
\node at (7,5) {};
\node at (7,6) {};
\node at (7,7) {};
\node at (7,8) {};
\node at (7,9) {};
\node at (7,10) {};
\node at (7,11) {};
\node at (7,12) {};
\node at (7,13) {};
\node at (7,14) {};
\node at (7,15) {};
\node at (8,5) {};
\node at (8,6) {};
\node at (8,7) {};
\node at (8,8) {};
\node at (8,9) {};
\node at (8,10) {};
\node at (8,11) {};
\node at (8,12) {};
\node at (8,13) {};
\node at (8,14) {};
\node at (8,15) {};
\node at (8,16) {};
\node at (9,7) {};
\node at (9,8) {};
\node at (9,9) {};
\node at (9,10) {};
\node at (9,11) {};
\node at (9,12) {};
\node at (9,13) {};
\node at (9,14) {};
\node at (9,15) {};
\node at (9,16) {};
\node at (10,9) {};
\node at (10,10) {};
\node at (10,11) {};
\node at (10,12) {};
\node at (10,13) {};
\node at (10,14) {};
\node at (10,15) {};
\node at (10,16) {};
\node at (11,11) {};
\node at (11,13) {};
\node at (11,15) {};
\node at (11,16) {};
\end{scope}
\draw[->,black] (0,0)--(\ax,\ay);
\draw[->,black] (0,0)--(\bx,\by);
\node[right] at  (\ax,\ay) {\(\alpha_1\)};
\node[above left] at (\bx,\by) {\(\alpha_2\)};
\end{tikzpicture}

%% file: a2gap00.tikz
\begin{tikzpicture}[scale=.15,>=latex]
	        \pgfmathsetmacro\ax{3}
	        \pgfmathsetmacro\ay{0}
	        \pgfmathsetmacro\bx{3* cos(120)}
	        \pgfmathsetmacro\by{3 * sin(120)}
\begin{scope}
\clip (-15,-5) rectangle (30,30);
            	\begin{scope}[cm={\ax,\ay,\bx,\by,(0,0)}]
            	\draw[style=help lines,gray] (-15,-15) grid (15,15);
            	\draw[cm={1,0,1,1,(0,0)},style=help lines, gray] (-15,-15) grid (15,15);
            	\end{scope}
\begin{scope}[cm={\ax,\ay,\bx,\by,(0,0)},every node/.style={circle,fill=black!50!green,inner sep=0pt,minimum size=4pt}]
\fill[black!50!green!50] (0,0)--(1,0)--(2,1)--(2,2)--(1,2)--(0,1)--(0,0);
\node at (0,0) {};
\node at (0,1) {};
\node at (0,3) {};
\node at (0,4) {};
\node at (0,6) {};
\node at (0,7) {};
\node at (0,9) {};
\node at (0,10) {};
\node at (0,12) {};
\node at (0,13) {};
\node at (0,15) {};
\node at (0,16) {};
\node at (0,18) {};
\node at (0,19) {};
\node at (1,0) {};
\node at (1,2) {};
\node at (1,3) {};
\node at (1,5) {};
\node at (1,6) {};
\node at (1,8) {};
\node at (1,9) {};
\node at (1,11) {};
\node at (1,12) {};
\node at (1,14) {};
\node at (1,15) {};
\node at (1,17) {};
\node at (1,18) {};
\node at (2,1) {};
\node at (2,4) {};
\node at (2,7) {};
\node at (2,10) {};
\node at (2,13) {};
\node at (2,16) {};
\node at (2,19) {};
\node at (3,0) {};
\node at (3,1) {};
\node at (3,3) {};
\node at (3,4) {};
\node at (3,6) {};
\node at (3,7) {};
\node at (3,9) {};
\node at (3,10) {};
\node at (3,12) {};
\node at (3,13) {};
\node at (3,15) {};
\node at (3,16) {};
\node at (3,18) {};
\node at (3,19) {};
\node at (4,0) {};
\node at (4,2) {};
\node at (4,3) {};
\node at (4,5) {};
\node at (4,6) {};
\node at (4,8) {};
\node at (4,9) {};
\node at (4,11) {};
\node at (4,12) {};
\node at (4,14) {};
\node at (4,15) {};
\node at (4,17) {};
\node at (4,18) {};
\node at (5,1) {};
\node at (5,4) {};
\node at (5,7) {};
\node at (5,10) {};
\node at (5,13) {};
\node at (5,16) {};
\node at (5,19) {};
\node at (6,0) {};
\node at (6,1) {};
\node at (6,3) {};
\node at (6,4) {};
\node at (6,6) {};
\node at (6,7) {};
\node at (6,9) {};
\node at (6,10) {};
\node at (6,12) {};
\node at (6,13) {};
\node at (6,15) {};
\node at (6,16) {};
\node at (6,18) {};
\node at (6,19) {};
\node at (7,0) {};
\node at (7,2) {};
\node at (7,3) {};
\node at (7,5) {};
\node at (7,6) {};
\node at (7,8) {};
\node at (7,9) {};
\node at (7,11) {};
\node at (7,12) {};
\node at (7,14) {};
\node at (7,15) {};
\node at (7,17) {};
\node at (7,18) {};
\node at (8,1) {};
\node at (8,4) {};
\node at (8,7) {};
\node at (8,10) {};
\node at (8,13) {};
\node at (8,16) {};
\node at (8,19) {};
\node at (9,0) {};
\node at (9,1) {};
\node at (9,3) {};
\node at (9,4) {};
\node at (9,6) {};
\node at (9,7) {};
\node at (9,9) {};
\node at (9,10) {};
\node at (9,12) {};
\node at (9,13) {};
\node at (9,15) {};
\node at (9,16) {};
\node at (9,18) {};
\node at (9,19) {};
\node at (10,0) {};
\node at (10,2) {};
\node at (10,3) {};
\node at (10,5) {};
\node at (10,6) {};
\node at (10,8) {};
\node at (10,9) {};
\node at (10,11) {};
\node at (10,12) {};
\node at (10,14) {};
\node at (10,15) {};
\node at (10,17) {};
\node at (10,18) {};
\node at (11,1) {};
\node at (11,4) {};
\node at (11,7) {};
\node at (11,10) {};
\node at (11,13) {};
\node at (11,16) {};
\node at (11,19) {};
\node at (12,0) {};
\node at (12,1) {};
\node at (12,3) {};
\node at (12,4) {};
\node at (12,6) {};
\node at (12,7) {};
\node at (12,9) {};
\node at (12,10) {};
\node at (12,12) {};
\node at (12,13) {};
\node at (12,15) {};
\node at (12,16) {};
\node at (12,18) {};
\node at (12,19) {};
\node at (13,0) {};
\node at (13,2) {};
\node at (13,3) {};
\node at (13,5) {};
\node at (13,6) {};
\node at (13,8) {};
\node at (13,9) {};
\node at (13,11) {};
\node at (13,12) {};
\node at (13,14) {};
\node at (13,15) {};
\node at (13,17) {};
\node at (13,18) {};
\node at (14,1) {};
\node at (14,4) {};
\node at (14,7) {};
\node at (14,10) {};
\node at (14,13) {};
\node at (14,16) {};
\node at (14,19) {};
\node at (15,0) {};
\node at (15,1) {};
\node at (15,3) {};
\node at (15,4) {};
\node at (15,6) {};
\node at (15,7) {};
\node at (15,9) {};
\node at (15,10) {};
\node at (15,12) {};
\node at (15,13) {};
\node at (15,15) {};
\node at (15,16) {};
\node at (15,18) {};
\node at (15,19) {};
\node at (16,0) {};
\node at (16,2) {};
\node at (16,3) {};
\node at (16,5) {};
\node at (16,6) {};
\node at (16,8) {};
\node at (16,9) {};
\node at (16,11) {};
\node at (16,12) {};
\node at (16,14) {};
\node at (16,15) {};
\node at (16,17) {};
\node at (16,18) {};
\node at (17,1) {};
\node at (17,4) {};
\node at (17,7) {};
\node at (17,10) {};
\node at (17,13) {};
\node at (17,16) {};
\node at (17,19) {};
\node at (18,0) {};
\node at (18,1) {};
\node at (18,3) {};
\node at (18,4) {};
\node at (18,6) {};
\node at (18,7) {};
\node at (18,9) {};
\node at (18,10) {};
\node at (18,12) {};
\node at (18,13) {};
\node at (18,15) {};
\node at (18,16) {};
\node at (18,18) {};
\node at (18,19) {};
\node at (19,0) {};
\node at (19,2) {};
\node at (19,3) {};
\node at (19,5) {};
\node at (19,6) {};
\node at (19,8) {};
\node at (19,9) {};
\node at (19,11) {};
\node at (19,12) {};
\node at (19,14) {};
\node at (19,15) {};
\node at (19,17) {};
\node at (19,18) {};
\end{scope}
\end{scope}
\draw[->,>=stealth,black] (0,0)--(\ax,\ay);
\draw[->,>=stealth,black] (0,0)--(\bx,\by);
\node[right] at  (\ax,\ay) {\(\alpha_1\)};
\node[above left] at (\bx,\by) {\(\alpha_2\)};
\end{tikzpicture}

%% file: a2gap11.tikz
\begin{tikzpicture}[scale=.15,>=latex]
	        \pgfmathsetmacro\ax{3}
	        \pgfmathsetmacro\ay{0}
	        \pgfmathsetmacro\bx{3* cos(120)}
	        \pgfmathsetmacro\by{3 * sin(120)}
\begin{scope}
\clip (-15,-5) rectangle (30,30);
            	\begin{scope}[cm={\ax,\ay,\bx,\by,(0,0)}]
            	\draw[style=help lines,gray] (-15,-15) grid (15,15);
            	\draw[cm={1,0,1,1,(0,0)},style=help lines, gray] (-15,-15) grid (15,15);
            	\end{scope}
\begin{scope}[cm={\ax,\ay,\bx,\by,(0,0)},every node/.style={circle,fill=black!50!green,inner sep=0pt,minimum size=4pt}]
\fill[black!50!green!50] (0,0)--(2,0)--(4,2)--(4,4)--(2,4)--(0,2)--(0,0);
\node at (0,0) {};
\node at (0,3) {};
\node at (0,6) {};
\node at (0,9) {};
\node at (0,12) {};
\node at (0,15) {};
\node at (0,18) {};
\node at (2,4) {};
\node at (2,7) {};
\node at (2,10) {};
\node at (2,13) {};
\node at (2,16) {};
\node at (2,19) {};
\node at (3,0) {};
\node at (3,6) {};
\node at (3,9) {};
\node at (3,12) {};
\node at (3,15) {};
\node at (3,18) {};
\node at (4,2) {};
\node at (4,5) {};
\node at (4,8) {};
\node at (4,11) {};
\node at (4,14) {};
\node at (4,17) {};
\node at (4,20) {};
\node at (5,4) {};
\node at (5,7) {};
\node at (5,10) {};
\node at (5,13) {};
\node at (5,16) {};
\node at (5,19) {};
\node at (6,0) {};
\node at (6,3) {};
\node at (6,6) {};
\node at (6,9) {};
\node at (6,12) {};
\node at (6,15) {};
\node at (6,18) {};
\node at (7,2) {};
\node at (7,5) {};
\node at (7,8) {};
\node at (7,11) {};
\node at (7,14) {};
\node at (7,17) {};
\node at (7,20) {};
\node at (8,4) {};
\node at (8,7) {};
\node at (8,10) {};
\node at (8,13) {};
\node at (8,16) {};
\node at (8,19) {};
\node at (9,0) {};
\node at (9,3) {};
\node at (9,6) {};
\node at (9,9) {};
\node at (9,12) {};
\node at (9,15) {};
\node at (9,18) {};
\node at (10,2) {};
\node at (10,5) {};
\node at (10,8) {};
\node at (10,11) {};
\node at (10,14) {};
\node at (10,17) {};
\node at (10,20) {};
\node at (11,4) {};
\node at (11,7) {};
\node at (11,10) {};
\node at (11,13) {};
\node at (11,16) {};
\node at (11,19) {};
\node at (12,0) {};
\node at (12,3) {};
\node at (12,6) {};
\node at (12,9) {};
\node at (12,12) {};
\node at (12,15) {};
\node at (12,18) {};
\node at (13,2) {};
\node at (13,5) {};
\node at (13,8) {};
\node at (13,11) {};
\node at (13,14) {};
\node at (13,17) {};
\node at (13,20) {};
\node at (14,4) {};
\node at (14,7) {};
\node at (14,10) {};
\node at (14,13) {};
\node at (14,16) {};
\node at (14,19) {};
\node at (15,0) {};
\node at (15,3) {};
\node at (15,6) {};
\node at (15,9) {};
\node at (15,12) {};
\node at (15,15) {};
\node at (15,18) {};
\node at (16,2) {};
\node at (16,5) {};
\node at (16,8) {};
\node at (16,11) {};
\node at (16,14) {};
\node at (16,17) {};
\node at (16,20) {};
\node at (17,4) {};
\node at (17,7) {};
\node at (17,10) {};
\node at (17,13) {};
\node at (17,16) {};
\node at (17,19) {};
\node at (18,0) {};
\node at (18,3) {};
\node at (18,6) {};
\node at (18,9) {};
\node at (18,12) {};
\node at (18,15) {};
\node at (18,18) {};
\node at (19,2) {};
\node at (19,5) {};
\node at (19,8) {};
\node at (19,11) {};
\node at (19,14) {};
\node at (19,17) {};
\node at (19,20) {};
\node at (20,4) {};
\node at (20,7) {};
\node at (20,10) {};
\node at (20,13) {};
\node at (20,16) {};
\node at (20,19) {};
\end{scope}
\end{scope}
\draw[->,>=stealth,black] (0,0)--(\ax,\ay);
\draw[->,>=stealth,black] (0,0)--(\bx,\by);
\node[right] at  (\ax,\ay) {\(\alpha_1\)};
\node[above left] at (\bx,\by) {\(\alpha_2\)};
\end{tikzpicture}

%% file: b2gap00.tikz
\begin{tikzpicture}[scale=.1,>=latex]
	        \pgfmathsetmacro\ax{3 * sqrt(2)}
	        \pgfmathsetmacro\ay{0}
	        \pgfmathsetmacro\bx{3 * cos(135)}
	        \pgfmathsetmacro\by{3 * sin(135)}
\clip (-20,-5) rectangle (30,30);
            	\begin{scope}[cm={\ax,\ay,\bx,\by,(0,0)}]
            	\draw[style=help lines,gray] (-20,-20) grid (30,30);
            	\draw[cm={1,0,1,1,(0,0)},style=help lines, gray] (-20,-20) grid (30,30);
            	\draw[cm={1,1,0,1,(0,0)},style=help lines, gray] (-20,-20) grid (30,30);
            	\draw[cm={1,2,1,0,(0,0)},style=help lines, gray] (-20,-20) grid (30,30);
            	\end{scope}
\begin{scope}[cm={\ax,\ay,\bx,\by,(0,0)},every node/.style={circle,fill=blue,inner sep=0pt,minimum size=4pt}]
\fill[blue!50] (0,0)--(1,0)--(2,1)--(3,3)--(3,4)--(2,4)--(1,3)--(0,1)--(0,0);
\node at (0,0) {};
\node at (0,1) {};
\node at (0,2) {};
\node at (0,3) {};
\node at (0,4) {};
\node at (0,5) {};
\node at (0,6) {};
\node at (0,7) {};
\node at (0,8) {};
\node at (0,9) {};
\node at (0,10) {};
\node at (0,11) {};
\node at (0,12) {};
\node at (0,13) {};
\node at (0,14) {};
\node at (0,15) {};
\node at (0,16) {};
\node at (0,17) {};
\node at (0,18) {};
\node at (0,19) {};
\node at (1,0) {};
\node at (1,3) {};
\node at (1,4) {};
\node at (1,5) {};
\node at (1,6) {};
\node at (1,7) {};
\node at (1,8) {};
\node at (1,9) {};
\node at (1,10) {};
\node at (1,11) {};
\node at (1,12) {};
\node at (1,13) {};
\node at (1,14) {};
\node at (1,15) {};
\node at (1,16) {};
\node at (1,17) {};
\node at (1,18) {};
\node at (1,19) {};
\node at (2,1) {};
\node at (2,4) {};
\node at (2,5) {};
\node at (2,6) {};
\node at (2,7) {};
\node at (2,8) {};
\node at (2,9) {};
\node at (2,10) {};
\node at (2,11) {};
\node at (2,12) {};
\node at (2,13) {};
\node at (2,14) {};
\node at (2,15) {};
\node at (2,16) {};
\node at (2,17) {};
\node at (2,18) {};
\node at (2,19) {};
\node at (3,2) {};
\node at (3,3) {};
\node at (3,4) {};
\node at (3,6) {};
\node at (3,7) {};
\node at (3,8) {};
\node at (3,9) {};
\node at (3,10) {};
\node at (3,11) {};
\node at (3,12) {};
\node at (3,13) {};
\node at (3,14) {};
\node at (3,15) {};
\node at (3,16) {};
\node at (3,17) {};
\node at (3,18) {};
\node at (3,19) {};
\node at (4,3) {};
\node at (4,4) {};
\node at (4,5) {};
\node at (4,6) {};
\node at (5,4) {};
\node at (5,5) {};
\node at (5,6) {};
\node at (5,7) {};
\node at (6,5) {};
\node at (6,6) {};
\node at (6,7) {};
\node at (6,8) {};
\node at (7,6) {};
\node at (7,7) {};
\node at (7,8) {};
\node at (7,9) {};
\node at (8,7) {};
\node at (8,8) {};
\node at (8,9) {};
\node at (8,10) {};
\node at (9,8) {};
\node at (9,9) {};
\node at (9,10) {};
\node at (9,11) {};
\node at (10,9) {};
\node at (10,10) {};
\node at (10,11) {};
\node at (10,12) {};
\node at (11,10) {};
\node at (11,11) {};
\node at (11,12) {};
\node at (11,13) {};
\node at (12,11) {};
\node at (12,12) {};
\node at (12,13) {};
\node at (12,14) {};
\node at (13,12) {};
\node at (13,13) {};
\node at (13,14) {};
\node at (13,15) {};
\node at (14,13) {};
\node at (14,14) {};
\node at (14,15) {};
\node at (14,16) {};
\node at (15,14) {};
\node at (15,15) {};
\node at (15,16) {};
\node at (15,17) {};
\node at (16,15) {};
\node at (16,16) {};
\node at (16,17) {};
\node at (16,18) {};
\node at (17,16) {};
\node at (17,17) {};
\node at (17,18) {};
\node at (17,19) {};
\node at (18,17) {};
\node at (18,18) {};
\node at (18,19) {};
\node at (19,18) {};
\node at (19,19) {};
\end{scope}
\draw[->,black] (0,0)--(\ax,\ay);
\draw[->,black] (0,0)--(\bx,\by);
\node[right] at  (\ax,\ay) {\(\alpha_1\)};
\node[above left] at (\bx,\by) {\(\alpha_2\)};
\end{tikzpicture}

%% file: b2gap24.tikz
\begin{tikzpicture}
[scale=.08,>=latex]
	        \pgfmathsetmacro\ax{3 * sqrt(2)}
	        \pgfmathsetmacro\ay{0}
	        \pgfmathsetmacro\bx{3 * cos(135)}
	        \pgfmathsetmacro\by{3 * sin(135)}
\clip (-30,-5) rectangle (50,50);
            	\begin{scope}[cm={\ax,\ay,\bx,\by,(0,0)}]
            	\draw[style=help lines,gray] (-20,-20) grid (30,30);
            	\draw[cm={1,0,1,1,(0,0)},style=help lines, gray] (-20,-20) grid (30,30);
            	\draw[cm={1,1,0,1,(0,0)},style=help lines, gray] (-20,-20) grid (30,30);
            	\draw[cm={1,2,1,0,(0,0)},style=help lines, gray] (-20,-20) grid (30,30);
            	\end{scope}
\begin{scope}[cm={\ax,\ay,\bx,\by,(0,0)},every node/.style={circle,fill=blue,inner sep=0pt,minimum size=4pt}]
\fill[blue!50] (0,0)--(3,0)--(8,5)--(11,11)--(11,16)--(8,16)--(3,11)--(0,5)--(0,0);
\node at (0,0) {};
\node at (0,5) {};
\node at (0,6) {};
\node at (0,7) {};
\node at (0,8) {};
\node at (0,9) {};
\node at (0,10) {};
\node at (0,11) {};
\node at (0,12) {};
\node at (0,13) {};
\node at (0,14) {};
\node at (0,15) {};
\node at (0,16) {};
\node at (0,17) {};
\node at (0,18) {};
\node at (0,19) {};
\node at (0,20) {};
\node at (0,21) {};
\node at (0,22) {};
\node at (0,23) {};
\node at (3,0) {};
\node at (3,11) {};
\node at (3,12) {};
\node at (3,13) {};
\node at (3,14) {};
\node at (3,15) {};
\node at (3,16) {};
\node at (3,17) {};
\node at (3,18) {};
\node at (3,19) {};
\node at (3,20) {};
\node at (3,21) {};
\node at (3,22) {};
\node at (3,23) {};
\node at (8,5) {};
\node at (8,16) {};
\node at (8,17) {};
\node at (8,18) {};
\node at (8,19) {};
\node at (8,20) {};
\node at (8,21) {};
\node at (8,22) {};
\node at (8,23) {};
\node at (9,6) {};
\node at (9,17) {};
\node at (9,18) {};
\node at (9,19) {};
\node at (9,20) {};
\node at (9,21) {};
\node at (9,22) {};
\node at (9,23) {};
\node at (10,7) {};
\node at (10,18) {};
\node at (10,19) {};
\node at (10,20) {};
\node at (10,21) {};
\node at (10,22) {};
\node at (10,23) {};
\node at (11,8) {};
\node at (11,11) {};
\node at (11,16) {};
\node at (11,17) {};
\node at (11,18) {};
\node at (11,20) {};
\node at (11,21) {};
\node at (11,22) {};
\node at (11,23) {};
\node at (12,9) {};
\node at (12,12) {};
\node at (12,17) {};
\node at (12,18) {};
\node at (12,19) {};
\node at (12,20) {};
\node at (13,10) {};
\node at (13,13) {};
\node at (13,18) {};
\node at (13,19) {};
\node at (13,20) {};
\node at (13,21) {};
\node at (14,11) {};
\node at (14,14) {};
\node at (14,19) {};
\node at (14,20) {};
\node at (14,21) {};
\node at (14,22) {};
\node at (15,12) {};
\node at (15,15) {};
\node at (15,20) {};
\node at (15,21) {};
\node at (15,22) {};
\node at (15,23) {};
\node at (16,13) {};
\node at (16,16) {};
\node at (16,21) {};
\node at (16,22) {};
\node at (16,23) {};
\node at (17,14) {};
\node at (17,17) {};
\node at (17,22) {};
\node at (17,23) {};
\node at (18,15) {};
\node at (18,18) {};
\node at (18,23) {};
\node at (19,16) {};
\node at (19,19) {};
\node at (20,17) {};
\node at (20,20) {};
\node at (21,18) {};
\node at (21,21) {};
\end{scope}
\draw[->,black] (0,0)--(\ax,\ay);
\draw[->,black] (0,0)--(\bx,\by);
\node[right] at  (\ax,\ay) {\(\alpha_1\)};
\node[above left] at (\bx,\by) {\(\alpha_2\)};
\end{tikzpicture}